\documentclass[a4paper]{amsart}
\usepackage{amssymb}
\usepackage[size=tiny]{todonotes}
\usepackage{xfrac} 
\usepackage{comment}
\usepackage{mathtools}
\usepackage{tikz-cd}

\usepackage[
    backend=bibtex,
    style=numeric,
    citestyle=nature,
    maxcitenames=9,
    maxbibnames=9,
    sorting=nyt,
    natbib=true,
    url=false, 
    doi=false,
    eprint=true
]{biblatex}
\newtheorem{theorem}{Theorem}[section] 
\newtheorem{lemma}[theorem]{Lemma}
\newtheorem{proposition}[theorem]{Proposition}

\newtheorem{corollary}[theorem]{Corollary}

\theoremstyle{definition}

\newtheorem{definition}[theorem]{Definition}
\newtheorem{remark}[theorem]{Remark}

\newcommand{\RR}{\mathbb{R}}
\newcommand{\QQ}{\mathbb{Q}}

\newcommand{\NN}{\mathbb{N}}
\newcommand{\PP}{\mathbb{P}}

\DeclareMathOperator{\vol}{vol}

\newcommand{\I}{\mathcal{I}}
\newcommand{\B}{\mathcal{B}}

\DeclareMathOperator{\Sym}{sym}
\DeclareMathOperator{\id}{id}

\DeclareMathOperator{\rk}{rk}

\title{Linear operators preserving volume polynomials}
\author{Lukas Grund}
\author{Hendrik S\"uss}
\address{Friedrich-Schiller-Universit\"at Jena, Fakult\"at f\"ur Mathematik und Informatik, Institut für Mathematik, Ernst-Abbe-Platz 2, 07743 Jena, Germany}
\email{lukas.grund@uni-jena.de}
\address{Friedrich-Schiller-Universit\"at Jena, Fakult\"at f\"ur Mathematik und Informatik, Institut für Mathematik, Ernst-Abbe-Platz 2, 07743 Jena, Germany}
\email{hendrik.suess@uni-jena.de}


\begin{document}

\begin{abstract}
  We study linear operators preserving the property of being a volume polynomial. More, precisely we show that a linear operator preserves this property if the associated \emph{symbol} is itself a volume polynomial. This can be seen as an analogue to theorems by Borcea-Br\"and\'en and Br\"and\'en-Huh for stable polynomials and Lorentzian polynomials, respectively.
\end{abstract}
\maketitle

\section{Introduction}\label{sec:introduction}
Given a  projective variety $X$ of dimension $d$ and nef  divisor classes $D_1,\ldots,D_n$ we call
  \[\vol_{X,D_1,\ldots,D_n}=\left(\sum x_i D_i \right)^d=\sum_{|\alpha|=d} \frac{d!}{\alpha_1! \cdots \alpha_n!}\left(D_1^{\alpha_1}\cdots D_n^{\alpha_n}\right)x_1^{\alpha_1}\cdots x_n^{\alpha_n}\]
  the \emph{associated volume polynomial}.  Here, $\left(D_1^{\alpha_1} \cdots D_n^{\alpha_n}\right)$ denotes the intersection product of the corresponding divisor classes. More generally, we also call  limits of positive multiples of such polynomials \emph{volume polynomials}. This definition also covers volume polynomials arising from convex bodies. This follows from the theory of toric varieties, see \cite[Section~5.4]{zbMATH00447302}. Volume polynomials come with strong log concavity properties, known as Khovanskii-Teissier or Alexandrov-Fenchel inequalities, which historically have been used to prove log concavity for sequences arising in combinatorics, see e.g. \cite{zbMATH03760173,zbMATH03893257}. There are, however, limitations to this approach. For example in matroid theory volume polynomials are rather tied to the representable case. Moreover, due to the geometric nature of the definition, proving that a polynomial is a volume polynomials usually involves some ad-hoc geometric construction. To overcome these limitations \emph{Lorentzian polynomials} were introduced by Br\"and\'en and Huh in \cite{BrandenHuh} and independently by Anari, Oveis Gharan and Vinzant in \cite{zbMATH07442560}. The definition of this class of polynomials is purely algebraic and in comes with a powerful theory of operators preserving this property, which often allows to show the Lorentzian property for a given polynomial by deriving it from a known  Lorentzian polynomial via these preserving operators in a more or less mechanical way.

  Although Lorentzian polynomials are a powerful tool, the approach breaks down, when one actually wants to show that a polynomial fulfils the strictly stronger property of being a volume polynomial. There are several situations where such a result would in fact be desired. For example it has been conjectured in \cite{Eur} that the independent set generating polynomial of a realisable matroid is a volume polynomial. 

  Our aim here is to provide a theory of preservers for volume polynomials and a toolbox of operators which can be used to derive new volume polynomials from known ones without having to care about geometric constructions.  In \cite{zbMATH05598918} the authors associated to an operator a polynomial, called the \emph{symbol} of the operator.  In \cite{BrandenHuh} it was shown that an operator with Lorentzian symbol preserves the Lorentzian property.
  Our main result replicates this theorem for volume polynomials.

  In Section~\ref{sec:preliminaries} we fix our notation and recall some  operations, which are well known to preserve the property of being a volume polynomial.

  In Section~\ref{sec:preservers} we prove our symbol criterion. As a proof of concept we reproduce and extend the results of \cite{ross2025diagonalizations} by using symbols instead of the ad-hoc arguments from \cite{ross2025diagonalizations}.  As a result, for many known preservers of Lorentzian polynomials we show that they also preserve volume polynomials.  

  Following \cite{RSW23} we introduce \emph{co-volume polynomials} as differential operators with constant coefficients preserving the property of being a volume polynomial.  In Section~\ref{sec:co-volume} we show that this notion coincides with the one introduced by Aluffi in \cite{Aluffi}.

  In Section~\ref{sec:matroids} we discuss the appearance of (conjectured) volume polynomials in matroid theory.
  \subsection*{Acknowledgements} We are grateful to Matt Larson for pointing out to us the applications to matroids, for further helpful comments on the manuscript and for patiently answering all our questions on matroid theory. We also would like thank Paolo Aluffi for helpful discussions on co-volume polynomials.   This material is based upon work supported by DFG grant 539864509.
  \section{Preliminaries}
  \label{sec:preliminaries}

  In this paper we denote the non-negative integers by $\NN$.  For an elements $\alpha = (\alpha_1,\ldots,\alpha_n) \in  \NN^n$ and $\beta = (\beta_1,\ldots,\beta_n) \in  \NN^n$ set $|\alpha|=\sum_i \alpha_i$, $\alpha!=\prod_i \alpha_i!$, $\binom{\alpha}{\beta}=\prod_i \binom{\alpha_i}{\beta_i}$ and $\binom{|\alpha|}{\alpha}=\frac{|\alpha|!}{\alpha!}$. 
We denote the $i$-th canonical basis vectors in $\NN^^n$ by $e_i$ and the element $(d,\ldots,d) \in \NN^n$ by $\underline{\mathbf{d}}$.
We write $x^\alpha$ for the monomial in $\RR[x_1,\ldots,x_n]$ with exponent vector $\alpha$, i.e. $x^\alpha = \prod_{i=1}^n x_i^{\alpha_i}$. Let $s \in \RR[x_1, \ldots x_n]$ be a homogeneous polynomial. Then we also consider the corresponding differential operator $\partial_s:=s(\partial_{x_1},\ldots,\partial_{x_n}) \in \RR[\partial_{x_1},\ldots,\partial_{x_n}]$. Moreover, for $\kappa \in \mathbb{N}^n$, we write $\RR[\underline x]$ as short-hand for $\RR[x_1\ldots,x_n]$ and  set \[\mathbb{R}_\kappa[\underline x] = \{ \text{polynomials in } \mathbb{R}[x_1,\dots,x_n] \text{ of degree at most } \kappa_i \text{ in } x_i \text{ for every }i \}\]

In the following we will denote the set of all volume polynomials of degree $d$ in $n$ variables by $V^d_n$.

The following facts about volume polynomials are well known. For proofs we refer e.g. to \cite[Section 2.2]{ross2025diagonalizations}.
\begin{lemma}
  \label{lem:linmap}
  Assume that $f(x_1\ldots,x_n)$ is a volume polynomial and $A$ is a $(n\times m)$-matrix with non-negative entries. 
  Then $f(Ax')$ is again a volume polynomial in variables $x_1',\ldots,x_m'$.
\end{lemma}

\begin{lemma}
  \label{lem:prodofvol}
  Products of volume polynomials are again volume polynomials.
\end{lemma}

\begin{lemma}
  \label{lem:monomisvol}
  Every monomial is a volume polynomial.
\end{lemma}

\begin{lemma}
  \label{lem:schubertisvol}
  Let $s \in \mathbb{R}[x_1,\dots,x_n]$ be a Schubert polynomial. The differential 
  operator $\partial_s$ preserves volume polynomials.\footnote{The argument in \cite[Section 2.2]{ross2025diagonalizations} was not sufficient. A correct proof (more generally for co-volume polynomials) is provided in \cite{grund2025linearoperatorspreservingvolume}}
\end{lemma}

\section{Operators preserving volume polynomials}
\label{sec:preservers}

Fix some $\kappa \in \mathbb{N}^n$ and $\gamma \in \mathbb{N}^m$. We fix a linear operator
\[T:\mathbb{R}_\kappa[\underline w] \to \mathbb{R}_\gamma[\underline w],\]
such that $T$ is homogeous of degree $l$ for some $l \in \mathbb{Z}$:
\[(0 \leq \alpha \leq \kappa \text{ and } T(w^\alpha) \neq 0) \implies \deg \ T(w^\alpha) = \deg \ w^\alpha + l.\]
We set $|\kappa| = k$. The symbol of $T$ is a homogeous polynomial of degree $k + l$ in $m+n$ variables definied by
\[\Sym_T(w,u) = \sum_{0 \leq \alpha \leq \kappa} \binom{\kappa}{\alpha} T(w^\alpha) u^{\kappa - \alpha}. \]
 
\begin{theorem}
\label{thm:volsym}
    The homogenous operator $T$ preserves volume polynomials if its symbol $\Sym_T$ is a volume polynomial, i.e. if $\Sym_T \in V_{m +n}^{k+l}$ and $f \in {V}_n^d \cap \mathbb{R}_\kappa[\underline w]$, then $T(f) \in V_m^{d+l}$.    
\end{theorem}

\begin{proof}
    Set $f = \sum_\alpha c_\alpha v^\alpha$. We define the following differential operators
    \[D_{u_i,v_i} \coloneqq \sum_{0 \leq \gamma_i \leq \kappa_i} \partial_{u_i}^{\gamma_i} \partial_{v_i}^{\kappa_i- \gamma_i} = s_{\kappa_i}(\partial_{u_i},\partial_{v_i}) = \partial_{s_{\kappa_i}}\]
    where $s_{\kappa_i}(\partial_{u_i},\partial_{v_i})$ is the complete homogeneous symmetric polynomial of degree $\kappa_i$ in $\partial_{u_i}$, $\partial_{v_i}$. We further set 
    \[ D \coloneqq \prod_{0 \leq i \leq n} D_{u_i,v_i} = \sum_{0 \leq \gamma \leq \kappa} \partial_u^{\kappa - \gamma} \partial_v^{\gamma}.\]
    Since $s_{\kappa_i}$ is a Schubert polynomial, by Lemma~\ref{lem:schubertisvol} the operators $D_{u_i,v_i}$ and therefore also the composition $D$ preserve the property of being a volume polynomial.
    
    Choose some $0 \leq \alpha \leq \kappa$ and $0 \leq \gamma \leq \kappa$, we observe
    \[
    \partial_u^{\kappa - \gamma} \partial_v^{\gamma}(u^{\kappa-\alpha} f(v))_{\mid v,u = 0} =
    \begin{cases}
      (\kappa-\alpha)!\alpha! c_\alpha  & \text{if } \alpha = \gamma \\
      0 & \text{otherwise}
    \end{cases}
     \]
    Hence
    \[D(u^{\kappa-\alpha} f(v))_{\mid v,u = 0} = 
    \left(\sum_{0 \leq \gamma \leq \kappa} \partial_u^{\kappa - \gamma} \partial_v^{\gamma}\right)\Big(u^{\kappa-\alpha} f(v)\Big)_{\mid v,u = 0} = (\kappa-\alpha)!\alpha! c_\alpha.
    \]

    Now 
    \begin{align*}
      &D\left(\Sym_T(w,u) f(v)\right)_{\mid v,u=0} \\
      &= D \left( \sum_{0 \leq \alpha \leq \kappa} \binom{\kappa}{\alpha } T(w^\alpha) u^{\kappa-\alpha} f(v) \right)_{\mid v,u = 0} \\
      &= \sum_{0 \leq \alpha \leq \kappa} \binom{\kappa}{\alpha } T(w^\alpha) D(u^{\kappa-\alpha} f(v))_{\mid v,u = 0}\\
      &= \sum_{|\alpha| = d} \binom{\kappa}{\alpha} T(w^\alpha) (\kappa - \alpha)!  \alpha! c_\alpha \\
      &= \kappa! \sum_{|\alpha| = d} T(c_\alpha w^\alpha) = \kappa! T(f).
    \end{align*}
    It follows that $T(f)$ is a volume polynomial.
\end{proof}

\begin{remark}
  The proof is based on the same idea as the one of \cite[Theorem~3.2]{BrandenHuh}. In particular, note that our operator $(D_{u_i,v_i})_{|v_i=u_i=0}$ is just the higher-degree analogue of the multi-affine operator in \cite[Lemma~3.3]{BrandenHuh}.
\end{remark}

\begin{definition}
  The normalisation operator $N$ on polynomials is the linear operator satisfying $N(x^\alpha) = \frac{1}{\alpha!} x^\alpha$.
\end{definition}

\begin{definition}
  \label{def:dual-vol}
  For an element $s \in \mathbb{R}_{\kappa}[\underline x]$ we set
  \[s^\vee \coloneqq N(x^\kappa s(x_1^{-1},\dots,x_n^{-1})).\]
  We say a homogeneous polynomial $s$ is a \emph{co-volume polynomial} if $s^\vee$ is a volume polynomial. 
\end{definition}
Note, that this notion does not depend on the choice of $\kappa$.

\begin{theorem}
  \label{thm:dualvol}
  Let $\mathbb{R}_\kappa[\underline x]$ be a homogeneous polynomial. Then the corresponding differential operator
  $\partial_s \coloneqq s(\partial_{x_1},\dots,\partial_{x_n})$ preserves volume polynomials if and
  only if $s$ is a co-volume polynomial.
\end{theorem}
\begin{proof}
  Consider any $\gamma \geq \kappa$. The symbol of $\partial_s$ with $s = \sum_\alpha \lambda_\alpha x^\alpha$ is given by
  \[\Sym_{\partial_s} = \sum_\alpha \lambda_\alpha \cdot \frac{\gamma!}{(\gamma-\alpha)!} \cdot \prod_i (x_i+ w_i)^{\gamma_i} \cdot \prod_i (x_i + w_i)^{- \alpha_i}.\]
  Hence with Lemma \ref{lem:linmap} the polynomial $\Sym_{\partial_s}$ is a volume polynomial if and only if the same is true for
  \begin{align*}
    \sum_{\alpha} \lambda_\alpha \frac{\gamma!}{(\gamma-\alpha)!} x^\gamma x^{-\alpha} &= \gamma! x^\gamma \sum_{\alpha}  \frac{\lambda_\alpha}{(\gamma-\alpha)!} x^{-\alpha} \\
    &= \gamma! N(x^\gamma \cdot s(x_1^{-1},\dots,x_n^{-1})),
  \end{align*}
  but this is precisely what we are assuming as $s$ is a co-volume polynomial. Thus, by Theorem~\ref{thm:volsym} the operator $\partial_s$ preserves the volume polynomial property.
  For the other direction assume that $\partial_s$ maps volume polynomials to volume polynomials. With Lemma \ref{lem:monomisvol} every monomial is a volume polynomial, by assumption we obtain the volume polynomial property also for
  \[\partial_s(x^\kappa) = \kappa! \cdot N(x^\kappa s(x_1^{-1},\dots,x_n^{-1})). \]
\end{proof}

\begin{proposition}
  \label{prop:prodofdual}
  If $f$ and $g$ are co-volume polynomials, so is their product $fg$.
\end{proposition}
\begin{proof}
  This follows with Theorem \ref{thm:dualvol}, as $\partial_{fg} = \partial_f \circ \partial_g$.
\end{proof}

\begin{corollary}[{\cite[Corollary~3.3]{ross2025diagonalizations}}]
  \label{cor:normisvol}
  The normalisation $N(f)$ of a volume polynomial $f$ is again a volume polynomial. 
\end{corollary}
\begin{proof}
  Choose $\kappa \in \mathbb{N}^n$. With Theorem \ref{thm:volsym} we have to show that the symbol
  \[\Sym_N(w,u) = \sum_{0 \leq \alpha \leq \kappa} \binom{\kappa}{\alpha} \frac{w^\alpha}{\alpha!} u^{\kappa - \alpha}
  = \prod_{j=1}^n \left( \sum_{0 \leq \alpha_j \leq \kappa_j} \binom{\kappa_j}{\alpha_j} \frac{w_j^{\alpha_j}}{\alpha_j!} u_j^{\kappa_j - \alpha_j} \right)\]
  is a volume polynomial. Since with Lemma~\ref{lem:prodofvol} the product of volume polynomials is again a volume polynomial, we only have 
  to show that each of the bivariate polynomials is a volume polynomial. But these are Lorentzian \cite[Proof of 3.7]{BrandenHuh} and bivariate rational Lorentzian
  polynomials are also volume polynomials \cite[4.9]{BrandenHuh}.
\end{proof}

\begin{corollary}[{\cite[Corollary~2.9]{ross2025diagonalizations}}]
  \label{cor:prodofdenomisvol}
  If $N(f)$ and $N(g)$ are volume polynomials, then  $N(fg)$ is also a volume polynomial.
\end{corollary}
\begin{proof}
  Suppose $f,g \in \mathbb{R}_\kappa[\underline w]$. We look at the linear operator
  \[T: \mathbb{R}_\kappa[\underline w] \to \mathbb{R}[\underline w], \ N(h) \mapsto N(hg). \]
  With Theorem \ref{thm:volsym} we have to show that its symbol
  \[\Sym_T(w,u) = \kappa! \sum_{0 \leq \alpha \leq \kappa} N(w^\alpha g) \frac{u^{\kappa-\alpha}}{(\kappa-\alpha)!}\]
  is a volume polynomial. We consider this symbol as the linear operator
  \[S:\mathbb{R}_\kappa[\underline w] \to \mathbb{R}[\underline w, \underline u],\quad  N(h) \mapsto \sum_{0 \leq \alpha \leq \kappa} N(w^\alpha h) \frac{u^{\kappa-\alpha}}{(\kappa-\alpha)!}.\]
  With Theorem \ref{thm:volsym}, we have to show that its symbol
  \[\Sym_S(w,u,v) = \kappa! \sum_{0 \leq \beta \leq \kappa} \sum_{0 \leq \alpha \leq \kappa} \frac{w^{\alpha + \beta}}{(\alpha + \beta)!} \frac{u^{\kappa-\alpha}}{(\kappa-\alpha)!}\frac{v^{\kappa-\beta}}{(\kappa-\beta)!}\]
  is a volume polynomial. Without the normalisation factors this is a product of complete homogeneous symmetric polynomials and these are co-volume polynomials. Hence with Proposition \ref{prop:prodofdual} the product is a co-volume polynomial. We observe
  \begin{align*}
    \left(\sum_{0 \leq \beta \leq \kappa} \sum_{0 \leq \alpha \leq \kappa} w^{\alpha + \beta}u^{\kappa-\alpha}v^{\kappa-\beta} \right)^\vee &= N \left(w^{2\kappa} u^\kappa v^\kappa \sum_{0 \leq \alpha,\beta \leq \kappa} w^{-\alpha - \beta}u^{-\kappa+\alpha}v^{-\kappa+\beta} \right) \\
    &= N \left( \sum_{0 \leq \alpha,\beta \leq \kappa} w^{2\kappa -\alpha - \beta}u^{\alpha}v^{\beta} \right) \\
    &= N \left( \sum_{0 \leq \kappa-\alpha,\kappa-\beta \leq \kappa} w^{\alpha + \beta}u^{\kappa-\alpha}v^{\kappa-\beta} \right)\\
    &= N \left( \sum_{0 \leq \alpha,\beta \leq \kappa} w^{\alpha + \beta}u^{\kappa-\alpha}v^{\kappa-\beta} \right)\\
    &=\sum_{0 \leq \beta \leq \kappa} \sum_{0 \leq \alpha \leq \kappa} \frac{w^{\alpha + \beta}}{(\alpha + \beta)!} \frac{u^{\kappa-\alpha}}{(\kappa-\alpha)!}\frac{v^{\kappa-\beta}}{(\kappa-\beta)!}.
  \end{align*}
  and conclude that $\Sym_S(w,u,v)$ is a volume polynomial.
\end{proof}

\begin{corollary}[{\cite[Corollary~2.10, Theorem~1.4]{ross2025diagonalizations}}]
  \label{cor:lowertrunc}
  Given a volume polynomial $f = \sum_\alpha c_\alpha x^\alpha$ and $\gamma \in \NN^n$
  \begin{enumerate}
  \item the antiderivative
    \(\int^\gamma f \coloneqq \sum_\alpha \frac{\alpha!}{(\alpha+\gamma)!} c_\alpha x^{\alpha+\gamma}\) and
  \item  the lower truncation $f_{\geq \alpha} = \sum_{\alpha \geq \gamma} c_\alpha x^\alpha$ 
  \end{enumerate}
  are again a volume polynomials.
\end{corollary}
\begin{proof}
  We have $\int^\gamma f = N(x^\gamma N^{-1}(f))$ and the first claim follows from Corollary \ref{cor:prodofdenomisvol}. For the second claim observer that $f_{\geq \gamma} = \int^\gamma \partial^\gamma f$.
\end{proof}

\begin{corollary}
  \label{cor:uppertrunc}
  If $f$ is a volume polynomial then the \emph{upper truncation}
\(f_{\leq \gamma} \coloneqq \sum_{\alpha \leq \gamma} c_\alpha x^\alpha\)
  is a volume polynomial.
\end{corollary}
\begin{proof}
  Suppose $f\in \mathbb{R}_\kappa[\underline w]$. We look at the linear operator
  \[T: \mathbb{R}_\kappa[\underline w] \to \mathbb{R}[\underline w], \ f \mapsto f_{\leq \gamma}. \]
  The symbol of this operator is
  \begin{align*}
    \Sym_T(w,u) &= \sum_{0 \leq \alpha \leq \gamma} \binom{\kappa}{\alpha} w^\alpha u^{\kappa - \alpha} \\
    &= \sum_{\kappa - \gamma \leq \kappa - \alpha \leq \kappa} \binom{\kappa}{\alpha} w^\alpha u^{\kappa - \alpha} \\
    &= \left( \sum_{0 \leq \kappa - \alpha \leq \kappa} \binom{\kappa}{\alpha} w^\alpha u^{\kappa - \alpha} \right)_{\geq (0,\kappa - \gamma)}\\
                &= \left( \sum_{0 \leq \alpha \leq \kappa} \binom{\kappa}{\alpha} w^{\kappa - \alpha} u^{\alpha} \right)_{\geq (0,\kappa - \gamma)}\\
    &=\left(\prod_i (w_i+u_i)^{\kappa_i}  \right)_{\geq (0,\kappa - \gamma)}                  
  \end{align*}
 This is a lower truncation of a product of linear forms and with Corollary \ref{cor:lowertrunc} a volume polynomial.
  Hence $T$ preserves volume polynomials.
\end{proof}

\begin{corollary}[{\cite[Theorem~1.1]{ross2025diagonalizations}}]
  \label{cor:diagonalisation}
  Assume that $N(f) \in \mathbb{R}_\kappa[\underline w]$ is a volume polynomial. Then the normalisation $N(g)$ of the \emph{diagonalisation} $g(w_1,w_3,\ldots,w_n)=f(w_1,w_1,w_3,\ldots,w_n)$ is again a volume polynomial.
\end{corollary}
\begin{proof}
We look at the linear operator
  \[T: \mathbb{R}_\kappa[\underline w] \to \mathbb{R}[\underline w], \ f \mapsto f(w_1,w_1,w_3,\dots,w_n). \]
  We need to prove that the operator $S = N \circ T \circ N^{-1}$ preserves the class of volume polynomials. The symbol of $S$ is
  \begin{align*}
    \Sym_S(w,u) &= \kappa_1!\kappa_2! \prod_{i=3}^n (w_i+u_i)^{\kappa_i} \\
    &\sum_{0 \leq \alpha_1 \leq \kappa_1} \sum_{0 \leq \alpha_2 \leq \kappa_2} \frac{w_1^{\alpha_1+\alpha_2}}{(\alpha_1+\alpha_2)!}\frac{u_1^{\kappa_1-\alpha_1}}{(\kappa_1-\alpha_1)!}\frac{u_2^{\kappa_2-\alpha_2}}{(\kappa_2-\alpha_2)!},
  \end{align*}
  see also the proof of Lemma~4.8 in \cite{branden2021lower}.  With the same argument as in the proof of Corollary \ref{cor:prodofdenomisvol} this polynomial is a volume polynomial. Thus with Theorem \ref{thm:volsym} the operator $S$ preserves volume polynomials.
\end{proof}

\begin{corollary}
\label{cor:deg-0-diff-operator}
  For $a \in \mathbb{R}_{\geq 0}$ the operator $(1+aw_j\partial_i)$ preserves the property of being a volume polynomial.
\end{corollary}
\begin{proof}
  We look at the linear operator
  \[T: \mathbb{R}_\kappa[\underline w] \to \mathbb{R}[\underline w], \ f \mapsto (1+aw_j\partial_i)f. \]
  and calculate its Symbol
  \begin{align*}
    \Sym_T(w,u) &= T(w+u)^\kappa = (w+u)^\kappa + aw_j (\partial_{w_i} (w+u)^\kappa) \\
    &=  (w+u)^\kappa + a\kappa_i \cdot w_j ( (w+u)^{\kappa - e_i}) =  (w+u)^{\kappa - e_i} (w_i + u_i + a\kappa_i \cdot w_j).
  \end{align*}
  and notice that this is a product of volume polynomials.
\end{proof}

We consider the \emph{polarisation operator}
\begin{align*}
  \Pi^\uparrow_t \colon \RR_{(\kappa,k)}[w_1,\ldots,w_n,t] &\to \RR_{(\kappa,\underline{\mathbf{1}})}[w_1,\ldots,w_n,t_1,\ldots,t_k]\\
  w^\alpha t^b &\mapsto {k \choose b}^{-1} w^\alpha e_b(t_1,\ldots, t_k)
\end{align*}
where $e_b$ is the elementary symmetric polynomial of degree $b$.
\begin{corollary}
  \label{cor:polarisation}
  The polarisation operator sends volume polynomials to volume polynomials.
\end{corollary}
\begin{proof}
  We calculate the symbol
  \[
    \sum_{\alpha, b} \binom{\kappa}{\alpha} w^\alpha e_b u^{\kappa-\alpha} s^{k-b} =(w+u)^\kappa \sum_{0 \leq b \leq k} e_b \cdot s^{k-b} =(w+u)^\kappa \prod_{1 \leq i \leq k} (t_i+s).\]
  This is a product of linear forms and hence a volume polynomial.
\end{proof}

The following operator is occurs in the context of the \emph{symmetric exclusion process}. It preserves stable and Lorentzian polynomials, see \cite[Cor.~3.9]{BrandenHuh}.
\begin{corollary}
  For $0 \leq \theta \leq 1$ the operator 
  $\Phi_{\theta}^{1,2} \colon \RR_{\underline{\mathbf 1}}[\underline w] \to \RR_{\underline{\mathbf 1}}[\underline w]$
  sending $f(w_1,\ldots,w_n)$ to \(\theta f(w_1,w_2,w_3,\ldots,w_n) + (1-\theta)f(w_2,w_1,w_3\ldots,w_n)\)
  preserves volume polynomials.
\end{corollary}
\begin{proof}
  The symbol of this operator is
  \[
    \big(\theta(w_1+u_1)(w_2+u_2) + (1-\theta)(w_2+u_1)(w_1+u_2)\big)\prod_{i=3}^n(w_i+u_i).
  \]
  The first factor is a volume polynomial by Lemma~\ref{lem:symbol-sym-exclusion} below. Hence, the symbol is a product of volume polynomials.
\end{proof}

\begin{lemma}
  \label{lem:symbol-sym-exclusion}
  For every $\theta \in \RR$ with $0 \leq \theta \leq 1$ the polynomial 
  \[f = xz + \theta(xw+yz) + (1-\theta)(zw+yx) +yw \in \mathbb{R}[x,y,z,w]\]
  is a volume polynomial.
\end{lemma}
\begin{proof}
  By approximation it is sufficient to show the claim  $\theta=\sfrac{p}{q} \in \QQ$.
  
  We choose $\bar{q} \in \mathbb{N}$ such that $\bar{q}^2 \geq q$.  We set $a \coloneqq \bar{q}^2 - q$, $b \coloneqq \bar{q}^2 - p$ and $c \coloneqq \bar{q}^2 - (q-p)$.  For symmetry reasons we may assume $\theta = \frac{p}{q} \leq \frac{1}{2}$. This is equivalent to $c = \bar{q}^2 - q + p \leq \bar{q}^2 - p = b$. So we have $b \geq c \geq a$.
  We choose disjoint index sets $I_b, I_{c-a}, J_{c-a}, I_{b-a}$ of cardinalities $b,c-a,c-a,b-a$ and a subset $I_a \subset I_b$ of cardinality $a$. 
  We set $r \coloneqq |I|$ where $I = I_b \sqcup I_{c-a} \sqcup J_{c-a} \sqcup I_{b-a}$. We now choose $r$ points in $\mathbb{P}^2$ in general position and denote the blow-up of $\mathbb{P}^2$ at those points by $X$. 
  Let $H \subset X$ be the pullback of a hyperplane  in $\mathbb{P}^2$ not meeting any of those points.
  We identify each $i \in I$ with an exceptional divisor $E_i$ in $X$. Define divisors
  \begin{align*}
    D_x &= \bar{q} H - \sum_{i \in I_b \sqcup  J_{c-a}} E_i \\
    D_w &= \bar{q} H - \sum_{i \in I_b \sqcup I_{c-a}} E_i \\
    D_y &= \bar{q} H - \sum_{i \in I_a \sqcup J_{c-a} \sqcup I_{b-a}} E_i \\
    D_z &= \bar{q} H - \sum_{i \in I_a \sqcup I_{c-a} \sqcup I_{b-a}} E_i.
  \end{align*}
  Note, that every sum above has exactly $b+c-a = \bar{q}^2 - p + \bar{q}^2 - q + p - \bar{q}^2 + q = \bar{q}^2$ summands. Hence,  $D_x,D_w,D_y$ and $D_z$ are nef by the known case of Nagata's conjecture, see \cite{zbMATH03305239}.
  We now have that
  \begin{align*}
    &(xD_x + yD_y + zD_z + wD_w)^2 = 2(\bar{q}^2 - a - c + a)xy \\
    &+ 2(\bar{q}^2 - a)xz + 2(\bar{q}^2 - b)xw  \\
    &+ 2(\bar{q}^2 - a - b + a)yz + 2(\bar{q}^2 - a)yw \\
    &+ 2(\bar{q}^2 - a - c + a)zw \\
    &= 2qxz + 2p(xw+yz) + 2(q-p)(zw+yx) + 2qyw
  \end{align*}
  is equal to $f$ after rescaling. 
\end{proof}

\section{An alternative characterisation of  co-volume polynomials}  
\label{sec:co-volume}

Given a product of projective spaces $\PP^\kappa \coloneqq \mathbb{P}^{\kappa_1} \times \dots \times \mathbb{P}^{\kappa_n}$ over an algebraically closed field $k$.
Let $H_i$ denote the pull-back of the hyperplane class from the $i$-th factor. Every class in the codimension-$d$ graded piece $A^d(\mathbb{P}^\kappa)$ of the Chow ring of $\mathbb{P}^\kappa$ may be written uniquely as 
\[z = \sum_{|\alpha|= d} c_{\alpha} H_{1}^{\alpha_1} \dots H_n^{\alpha_n} \cap  [\mathbb{P}^\kappa],\]
where $\alpha \leq \kappa$ and $c_\alpha \in \mathbb{Z}$. We associate with $z$ the polynomial
\[P_{z}(x_1,\dots,x_n) \coloneqq \sum_{|\alpha|=d} c_\alpha x^\alpha.\]

In \cite{Aluffi} Aluffi defined co-volume polynomials as limits of polynomials of the form $cP_{[W]}$ for a positive real number $c$ and a closed subvariety $W$ of a product of projective spaces. They also appeared as \emph{multidegree polynomials} in \cite{zbMATH07638477}. Our aim here is to show that our definition is consistent with Aluffi's.

\begin{theorem}
  A polynomial is a co-volume polynomial in our sense, i.e. fulfils the condition of Definition~\ref{def:dual-vol} if and only if it is a co-volume polynomial in Aluffi's sense, i.e. is a limit of polynomials of the form $cP_{[W]}$.
\end{theorem}

\begin{proof}
  One direction has been shown in \cite[Prop.~2.8]{Aluffi}.

  For the other direction we essentially invert the argument given there, but additionally we need to show that one is allowed  to cancel monomial factors from polynomials of the form $cP_{[W]}$ and still obtains a polynomial of that type. This will be done in Lemma~\ref{lem:covol-product-w-monomial}.

  Given polynomial $f$ of degree $d$, such that $g = x^{\underline{\mathbf d}} f(\sfrac{1}{x_1},\ldots,\sfrac{1}{x_n})$ is a volume polynomial. Then there is a sequence of polynomials
  \[g_i = \vol_{X_i, \lambda_1^i D_1^i,\dots, \lambda_n^i D_n^i}\]
  such that $D_j^i$ is a very ample divisor in the projective variety $X_i$. We get a closed immersion
  \[X_i \xhookrightarrow[]{\varphi} \mathbb{P} = \mathbb{P}^{\kappa_1^i} \times \dots \times \mathbb{P}^{\kappa_n^i}.\]
  such that $D_j^i = \varphi^*(H_j^i)$ where $H_j^i$ is a hyperplane in $\mathbb{P}^{\kappa_j^i}$. Let $h_i$ be the co-volume polynomial defined by $X_i$ in $\mathbb{P}$ and $f_i(x_1,\ldots, x_n)\coloneqq h_i(\lambda_1^{-1}x_1,\ldots, \lambda_n^{-1}x_n)$. Then
  \[\partial_{f_i} ( \lambda^{\kappa^i} x^{\kappa^i}) = g_i = N(x^{\kappa^i}f_i(\sfrac{1}{x_1},\ldots,\sfrac{1}{x_n})).\] 
  Now
  \[f_i = x^{\kappa^i - \underline{\mathbf d}} \cdot f_i'\]
  and with Lemma \ref{lem:covol-product-w-monomial} the factor $f_i'$ is still a co-volume polynomial. We get
  \[ g_i=N(x^{\underline{\mathbf d}}f_i'(\sfrac{1}{x_1},\ldots,\sfrac{1}{x_n})) \to g\]
  or equivalently $f_i' \to x^{\underline{\mathbf d}}N^{-1}(g)(\sfrac{1}{x_1},\ldots,\sfrac{1}{x_n}) = f$.
\end{proof}

\begin{lemma}
  \label{lem:covol-product-w-monomial}
  Fix $\PP=\PP^\kappa$. Then for every $f \in \RR_{\kappa-e_1}[\underline x]$ the class
  $f(H_1,\ldots,H_n)$ is the class of an irreducible subvariety if and only if the same is true for the class $H_1 \cdot f(H_1,\ldots,H_n)$.
\end{lemma}

\begin{proof}
  The ``only if'' part is clear by intersecting with the pullback of a generic hyperplane in $\PP^{\kappa_1}$.
  Let $X$ be a irreducible closed subvariety in $\mathbb{P}^\kappa$ with $P_{[X]} = x_1 \cdot f$. We want to show that $f$ is represented by an irreducible subvariety. We choose a point $p \in \PP^{\kappa_1}$ such that $p$ is not in $X$. Let $\psi \colon B \to \PP^{\kappa_1}$  be the blow-up in $p$ this is a $\PP^1$-bundle over $\PP^{\kappa_1-1}$ with structure morphism  $\pi \colon B \to \PP^{\kappa_1-1}$. Then the rational map $\pi\circ \psi^{-1}$ is the projection away from $p$.  We obtain the following diagram
  \[
    \begin{tikzcd}
      \bar{B} = B \times \mathbb{P}^{\kappa_2} \times \dots \times \mathbb{P}^{\kappa_n} \arrow[d, "\bar\psi=\psi\times\id"'] \arrow[rd, "\bar \pi=\pi \times \id"] &                                                                   \\
      \mathbb{P}^{\kappa_1} \times \dots \times \mathbb{P}^{\kappa_n} \arrow[r, dashed]                               & \mathbb{P}^{\kappa_1-1} \times \dots \times \mathbb{P}^{\kappa_n}
      \end{tikzcd}
  \]
  where $\bar\psi$ is a map of smooth projective varieties and $\bar\pi$ is proper. Hence the morphism of cycles $\bar\psi^*$ and $\bar\pi_*$ are defined.
  Using \cite[9.3]{Eisenbud_Harris_2016} we have that the Chow ring of $\overline{B}$ is a tensor product of the Chow rings of its factors. The Chow ring of $B$ is described in \cite[Prop. 2.13]{Eisenbud_Harris_2016}.
  Let $\Lambda'$ in $\PP^\kappa$ be the pullback of a hyperplane in $\PP^{\kappa_1}$ not containing $p$ and $\Lambda = \bar\psi^{-1}(\Lambda')$. 
  We choose a monomial $H^\alpha \in A^*(\PP^\kappa)$ and represent it with a product of linear subspaces $[L_1 \times \cdots \times L_n]$ such that $\dim L_i= \kappa_i- \alpha_i$ and   $L_1 \times \PP^{\kappa_2} \times \cdots \times \PP^{\kappa_n} \subset \Lambda'$ for $\alpha_1 > 0$. We set $\Lambda_1$ for the preimage of $L_1$ in $B$. Then
  \begin{align*}
    \bar\pi_* \bar\psi^*(H^\alpha) &= \bar\pi_* \bar\psi^*([L_1 \times \cdots \times L_n]) = \bar\pi_*([\Lambda_1 \times L_2 \times \cdots \times L_n]) \\
    &= \begin{cases}
      [L_1 \times \cdots \times L_n] = H^{\alpha-e_1}, & \text{if } \dim L_1 = \kappa_1 -\alpha_1 < \kappa_1 \\
      \bar\pi_*([B \times L_2 \times \cdots \times L_n]) = 0, & \text{if }  \dim L_1 = \kappa_1.
    \end{cases} 
  \end{align*} 
  In the first case we use that we have an have an isomorphism $\bar\pi_{\mid \Lambda}: \Lambda \to \PP^{\kappa-e_1}$, compare with \cite[2.1.9]{Eisenbud_Harris_2016}. In the second case we use that the dimension of \[\bar\pi(B \times L_2 \times \cdots \times L_n) = \PP^{\kappa_1-1} \times L_2 \times \cdots \times L_n\] drops. 
  Assume $[X] = \sum_{|\alpha| =d} c_\alpha H^\alpha \cap [\PP^\kappa]$ and $Y = \bar\pi(\bar\psi^{-1}(X))$. Then $\bar\pi$ is birational and flat when restricted to $\bar\psi^{-1}(X)$. It follows 
  \[[Y] = \bar\pi_* \bar\psi^*[X] = \sum_{|\alpha| =d, \alpha_1 \geq 1} c_\alpha H^{\alpha- e_1} \cap [\PP^{\kappa-e_1}]. \]
  Hence $f=P_{[Y]}$.
\end{proof}

\begin{remark}
  The operators preserving dually Lorentzian polynomials from \cite{RSW23} in fact all preserve co-volume polynomials. This can be shown with the same proofs as in \cite{RSW23} where Lorentzian polynomials have to be replaced with volume polynomials and dually Lorentzian polynomials with co-volume polynomials, respectively.
\end{remark}

\section{Volume polynomials in matroid theory}
\label{sec:matroids}

Given a matroid $M$ of rank $r$ together with its set of bases $\B$ and its set $\I$ of independent sets we consider the corresponding generating polynomials
\[
B_M = \sum_{\beta \in \B} x^\beta, \qquad I_M=\sum_{\iota \in \I} x^\iota y^{r-|\iota|}
\]
If $M$ is realised by a set of vectors $v_1,\ldots,v_n \in V$ then by \cite{zbMATH06547829} (see also \cite{rhrle2024logarithmic}), the bases generating polynomial $B_M$ is a volume polynomial (by dualisation it is also a co-volume polynomial). The same has been conjectured for the $I_M$ in  \cite[Conjecture 5.6.]{Eur}.

In general, it seems natural to expect Lorentzian polynomials appearing as generating polynomials in matroid theory to be actually volume polynomials in the realisable case.

Let us discuss two instances of this principle.

\subsection*{Diagonalised independent set polynomials of  delta-matroids}
Following \cite{rhrle2024logarithmic} we may consider the matroid $M_{+m}$ corresponding to the set of vectors \[v_1,\ldots, v_n,v_{n+1},\ldots, v_{n+m} \in V,\] where $v_{n+1},\ldots,v_{n+m} \in V$ are chosen generically. Then $B_{M_{+m}}(x_1,\ldots,x_{n+m})$  is a volume polynomial. Hence, by Lemma~\ref{lem:linmap} also \[B_{M_{+m}}(x_1,\ldots,x_n,\underbrace{m^{-1}y,\ldots, m^{-1}y}_m) = \sum_{\iota \in \I} \binom{m}{r-|\iota|} m^{|\iota|-r} x^\iota y^{r-|\iota|}\]
is a volume polynomial and therefore the same is true for
\[N(I_M) = \lim_{m \to \infty} B_{M_{+m}}(x_1,\ldots,x_n,m^{-1}y,\ldots, m^{-1}y).\]
In \cite{larson2023rank} so-called \emph{delta-matroids} were considered. In the proof of  \cite[Theorem~3.15]{larson2023rank} the author shows that the diagonalised generating polynomial for the independent sets in a delta-matroid arises from the independent set generating polynomial of an \emph{enveloping matroid} via diagonalisation and truncation. Hence, together with Theorem~\ref{cor:diagonalisation} and Proposition~\ref{cor:uppertrunc} we conclude that the diagonalised independent set generating polynomial of a realisable delta-matroid is a denormalised volume polynomial (and by dualisation of matroids also a co-volume polynomial). Moreover, if  \cite[Conjecture 5.6.]{Eur} holds, then the diagonalised independent set generating polynomial of a realisable delta-matroid would be a volume polynomial, as well.

\subsection*{Twisted K-polynomials and polymatroids} Given matroid $M$ with rank function $\rk_M$ on a base set $E$. Consider subsets $S_1 \ldots, S_n \subset E$. Then the \emph{restriction polymatroid} $P$ of $M$ to $S_1 \ldots, S_n$ is given by the rank function
\(\rk_P(I)=\rk_M\left(\bigcup_{i\in I} S_i\right)\) for $I \subset \{1,\ldots,n\}$.
We will not state the definition of a polymatroid, but note that every polymatroid (given by its rank function) arises this way.

Consider a multiplicity-free subvariety $X \subset \PP^\kappa$, i.e. the coefficients of the corresponding co-volume polynomial $f=P_{[X]}$ are $0$ or $1$. Then every monomial of the corresponding volume polynomial $f^\vee=\partial_f(x^\kappa)$ has the form $\frac{x^\alpha}{\alpha!}$. Then $f^\vee$ is the exponential basis generating function of some polymatroid $P$, see e.g. \cite[Prop.~5.1]{zbMATH07638477}. Alternatively this polymatroid maybe described by the rank function $\rk_P(I)=\dim p_I(X)$, where $p_I \colon \PP^\kappa \to \prod_{i \in I} \PP^{\kappa_i}$ is the projection. This follows from \cite[Thm.~A]{zbMATH07638477}.

Note, that by \cite[Prop.~7.15]{castillo2024kpolynomialsmultiplicityfreevarieties} the exponential basis generating function of every \emph{realisable} polymatroid arises this way.

Now, consider the \emph{homogenised twisted K-polynomial} \[g(\underline{x},t)=\mathcal{K}(t-x_1,\ldots, t-x_n)=\sum_{\alpha,b} c_\alpha x^\alpha t^b,\] which is characterised (in the realisable case) by the following identity of $K$-classes in the Grothendieck ring $K(\PP^\kappa)$.
\[[\mathcal O_X]=\sum_\alpha c_\alpha [\mathcal O_{H_1}]^{\alpha_1}\cdots [\mathcal O_{H_1}]^{\alpha_n}\]
 In Theorem~1.4 of \cite{eur2024ktheoreticpositivitymatroids} the authors show that under some extra condition $g^\vee$ can be derived by operators of the type considered in Corollary~\ref{cor:deg-0-diff-operator} from the exponential basis generating function of $P$. In our setup the extra condition means that the projection $X \to \PP^{\kappa_i}$ is generically finite for at least one of the factors. Then it follows that in this situation the homogenised twisted K-polynomial of a multiplicity-free subvariety $X \subset \PP^\kappa$ is a co-volume polynomial. Similarly in the case of $P$ being a matroid, i.e. $\kappa=\underline{\mathbf 1}$, the authors show in \cite[Remark~3.5]{eur2024ktheoreticpositivitymatroids} that (without any extra condition) the polynomial  $g^\vee$ can be derived from a polynomial, which is known to be a volume polynomial in the realisable case via operation from Lemmata~\ref{lem:linmap} and \ref{lem:schubertisvol}. Hence, also for a multiplicity-free subvariety in $\PP^{\underline{\mathbf 1}}$ the twisted K-polynomial turns out to be a volume polynomial.
\printbibliography
\end{document}